\documentclass[11pt, letterpaper,reqno]{amsart}
\usepackage{amsmath}
\usepackage{amsthm}
\usepackage{amsrefs}
\usepackage{amssymb}
\numberwithin{equation}{section}
\usepackage{float}
\usepackage{graphicx}

\newtheorem{corollary}{Corollary}[section]
\newtheorem{lemma}{Lemma}[section]
\newtheorem{theorem}{Theorem}[section]

\theoremstyle{definition}
\newtheorem{remark}{Remark}[section]

\DeclareMathOperator{\arcsinh}{arcsinh}

\begin{document}
\title{On the Hardy number of comb domains}

\author{Christina Karafyllia}  
\address{Institute for Mathematical Sciences, Stony Brook University, Stony Brook, NY 11794, U.S.A.}
\email{christina.karafyllia@stonybrook.edu}



\begin{abstract} Let ${H^p}\left( \mathbb{D} \right)$ be the Hardy space of all holomorphic functions on the unit disk $\mathbb{D}$ with exponent $p>0$. If $D\ne \mathbb{C}$ is a simply connected domain and $f$ is the Riemann mapping from $\mathbb{D}$ onto $D$, then the Hardy number of $D$, introduced by Hansen, is the supremum of all $p$ for which $f \in {H^p}\left( \mathbb{D} \right)$. Comb domains are a well-studied class of simply connected domains that, in general, have the form of the entire plane minus an infinite number of vertical rays. In this paper we study the Hardy number of a class of comb domains with the aid of the quasi-hyperbolic distance and we establish a necessary and sufficient condition for the Hardy number of these domains to be equal to infinity. Applying this condition, we derive several results that show how the mutual distances and the distribution of the rays affect the finiteness of the Hardy number. By a result of Burkholder our condition is also necessary and sufficient for all moments of the exit time of Brownian motion from comb domains to be infinite.
\end{abstract}

\maketitle

\section{Introduction}\label{int}

A classical problem in geometric function theory is to find geometric conditions for a holomorphic function on the unit disk to belong in Hardy spaces (see e.g. \cite{Ba}, \cite{Han1}, \cite{Han2}, \cite{Kar1}, \cite{Kim}, \cite{Pog} and \cite{Ra}). In this paper we study this problem in the case of conformal mappings from the unit disk onto a comb domain. The Hardy space with exponent $p>0$ \cite[p.\ 1-2]{Dur} is denoted by ${H^p}\left( \mathbb{D} \right)$ and is defined to be the set of all holomorphic functions, $f$, on the unit disk $\mathbb{D}$ that satisfy the condition 
\[\mathop {\sup }\limits_{0 < r < 1} \int_0^{2\pi } {{{| {f( {r{e^{i\theta }}} )} |}^p}d\theta  <  + \infty } .\]
The fact that a function $f$ belongs to some ${H^p}\left( \mathbb{D} \right)$ imposes a restriction on its growth and this restriction is stronger as $p$ increases. That is, if $p>q$ then $H^p (\mathbb{D}) \subset H^q (\mathbb{D})$. 
 
In \cite{Han1} Hansen studied the problem of determining the numbers $p$ for which $f \in {H^p}\left( \mathbb{D} \right)$ by studying $f\left( \mathbb{D} \right)$. For this purpose he introduced a number which he called the Hardy number of a region. Since we study comb domains, we only state the definition in the case of simply connected domains. Let $D\ne \mathbb{C}$ be a simply connected domain and $f$ be a Riemann mapping from $\mathbb{D}$ onto $D$. The Hardy number of $D$, or equivalently of $f$, is defined by 
\[{\rm h}\left( D \right) = \sup \left\{ {p > 0:f \in {H^p}\left( \mathbb{D} \right)} \right\}.\]
We note that this definition is independent of the choice of the Riemann mapping onto $D$. It is known that every conformal mapping on $\mathbb{D}$ belongs to ${H^p}\left( \mathbb{D} \right)$ for all $p \in (0,1/2)$ \cite[p.\ 50]{Dur}. This implies that  ${\rm h}\left( D \right)$ lies in $[1/2, +\infty]$. 

There is no general method for computing the Hardy number but there are some ways to estimate it for certain types of domains. In \cite{Han1} Hansen gave a lower bound for the Hardy number of an arbitrary region and improved this bound for simply connected domains. Moreover, he determined the exact value of the Hardy number of starlike \cite{Han1} and spiral-like regions \cite{Han2}. In \cite{Co2} Poggi-Corradini studied the Hardy number of K{\oe}nigs mappings. He also proved \cite{Pog} for a certain class of functions, which give a geometric model for the self-mappings of $\mathbb{D}$, that the Hardy number is equal to infinity if and only if the image region does not contain a twisted sector. Furthermore, in \cite{Ess} and \cite{Kim} Ess{\'e}n, and  Kim and Sugawa, respectively, gave a description of the Hardy number of a plane domain in terms of harmonic measure. In \cite{Kar} the current author gave a formula for the Hardy number of a simply connected domain in terms of hyperbolic distance. Finally, Burkholder \cite{Bur} studied the Hardy number of a domain in relation with the exit time of Brownian motion (see also \cite{Betsakos}). More precisely, if $D$ is a simply connected domain, then we define the number ${\widetilde h}(D)$ to be the supremum of all $p>0$ for which the $p$-th moment of the exit time of Brownian motion is finite. Then Burkholder proved in \cite{Bur} that 
\begin{equation}\label{brownian}
{\widetilde h}(D)={\rm h}(D)/2.
\end{equation}

Comb domains furnish an interesting class of simply connected domains and thus they have been studied from various points of view. For example, they have been studied in relation with the angular derivative (see \cite{Jen}, \cite{Karam} and references therein), the harmonic measure \cite{Bet} and the semigroups of holomorphic functions \cite{Betsa}. Moreover, in \cite{Bou} Boudabra and Markowsky studied the moments of the exit time of planar
Brownian motion from comb domains. 

Let $\left\{ {x_n } \right\}_{n \in \mathbb{Z}} $ be a strictly increasing sequence of real numbers such that $x_0=0$ and 
\[ \mathop {\inf }\limits_{n \in \mathbb{Z}} (x_{n}  - x_{n-1} )>0.\]
Also, let $\left\{ {c_n } \right\}_{n \in \mathbb{Z}} $ be a sequence of positive numbers such that for some constants $m,M>0$,
\[m \le c_n\le M\]
for every $n\in \mathbb{Z}$. We consider comb domains of the form (see Fig.\ \ref{comb})
\[D=\mathbb{C}\backslash \bigcup\limits_{n \in \mathbb{Z}} {\left\{ {x_n  + iy: |y| \ge c_n} \right\}}. \] 

\begin{figure} 
	\begin{center}
		\includegraphics[scale=0.5]{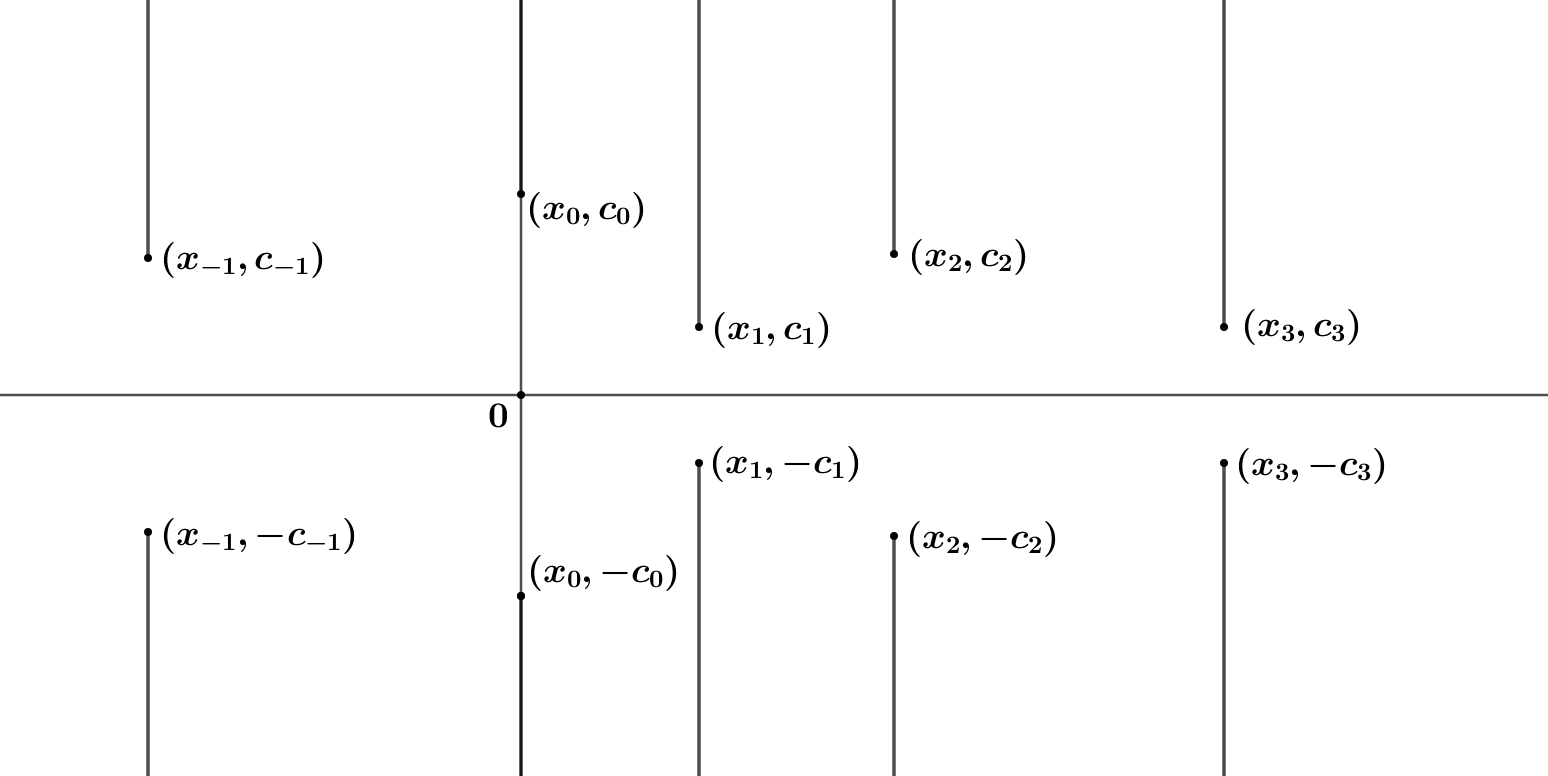}
		\vspace*{2mm}
		\caption{Comb domain}
		\label{comb}
	\end{center}
\end{figure}

Since we want to find conditions for the Hardy number of such comb domains to be equal to infinity, we can simplify the problem in the following way. First, we observe that if 
\[D_{m}=\mathbb{C}\backslash \bigcup\limits_{n \in \mathbb{Z}} {\left\{ {x_n  + iy: |y| \ge m} \right\}}\,\,{\text{and}}\,\,D_{M}=\mathbb{C}\backslash \bigcup\limits_{n \in \mathbb{Z}} {\left\{ {x_n  + iy: |y| \ge M} \right\}},\]
then 
$D_{m}\subseteq D\subseteq D_{M}$ and hence ${\rm h}(D_{M}) \le {\rm h}(D)\le {\rm h}(D_{m})$ (see \cite{Han1}). Moreover, since the Hardy number is invariant under affine mappings (see \cite{Han1}), we have ${\rm h}(D_{m})={\rm h}(D_{M})={\rm h}(D)$. Therefore, it suffices to study comb domains of the form
\[D_c=\mathbb{C}\backslash \bigcup\limits_{n \in \mathbb{Z}} {\left\{ {x_n  + iy: |y| \ge c} \right\}}, \]
where $c>0$. However, we can do more simplifications. We observe that if 
\[D_c^+=D\cap \left\{ z:\mathrm{Re} z>-x_1 \right\}\,\,{\text{and}}\,\,D_c^-=D\cap \left\{ z:\mathrm{Re} z<-x_{-1} \right\},\]
then ${\rm h}(D_c)=+\infty$ if and only if ${\rm h}(D_c^+)=+\infty$ and ${\rm h}(D_c^-)=+\infty$. This follows from Proposition 8 in \cite{Bou} and (\ref{brownian}). Furthermore, ${\rm h}(D_c^-)={\rm h}(-D_c^-)$. Therefore, it suffices to study the Hardy number of comb domains of the form $D_c^+$. Finally, since the Hardy number is invariant under affine mappings, without loss of generality, we suppose that $c=1$ and the infimum of the differences $x_n-x_{n-1}$ is greater than $1$. So, henceforth we consider comb domains $C$ of the form 
\begin{align}\label{des}
C=\left\{ z:\mathrm{Re} z>-x_1 \right\} \backslash \bigcup\limits_{n \in \mathbb{N}\cup \left\{0\right\}} \left\{ x_n+iy:|y|\ge 1 \right\},
\end{align} 
where $x_0=0$ and $\left\{ {x_n } \right\}_{n \in \mathbb{N}} $ is a strictly increasing sequence of positive numbers such that  
\[\mathop {\lim }\limits_{n \to  + \infty } x_n  =  + \infty\,\,\, {\rm and}\,\,\, \mathop {\inf }\limits_{n \in \mathbb{N}} \left( {x_{n}  - x_{n-1} } \right)>1.\]
 
First, we establish a necessary and sufficient condition for ${\rm{h}}(C)$ to be equal to infinity by studying the Euclidean distances between the rays. For every $n\in \mathbb{N}$, we denote these distances by 
\[\alpha_n=x_n-x_{n-1}.\]
 
\begin{theorem}\label{th2}
	Let $C$ be a comb domain of the form {\rm (\ref{des})}. Then ${\rm h}(C)=+\infty$ if and only if
	\[\mathop {\lim }\limits_{n \to  + \infty }\frac{{\sum \limits_{i = 1}^{n} \log  {{{ \alpha_i }}}}}{{\log x_{n} }}=+\infty\,\,\,or,\, equivalently,\,\,\,\mathop {\lim }\limits_{n \to  + \infty } \frac{{\sum \limits_{i = 1}^{n} \log  {{{  {\alpha_{i} } }}}}}{{\log \sum \limits_{i = 1}^{n}  {{{  {\alpha_{i} } }}} }}=+\infty.\]
\end{theorem}
An immediate consequence is that if the sequence $\alpha_n$ is bounded then ${\rm h}(C)=+\infty$. So, we actually study the case of $\alpha_n$ being unbounded. By applying Theorem \ref{th2} we can examine how the mutual distances and the distribution of the rays affect the finiteness of the Hardy number. First, we consider the case of $\alpha_n$ growing at a subexponential rate and prove that ${\rm h}(C)$ is always equal to infinity. 

\begin{theorem}\label{th3}
If 
\[\mathop {\lim }\limits_{n \to  + \infty } \frac{{\log \alpha_n }}{n} = 0,\]
then ${\rm h}(C)=+\infty$.
\end{theorem}
This result is stronger than the corollary of the main theorem of Boudabra and Markowsky in \cite{Bou}, where they approach the problem by studying the moments of the exit time of the Brownian motion. In fact, their main theorem implies that if $\alpha_n$ grows at most polynomially in $n$ then the Hardy number is infinite. However, it does not cover all subexponential sequences. For the proof of Theorem \ref{th3} see Section \ref{proofs}.

Next, we explain why the assumption in Theorem \ref{th3} cannot be relaxed. Theorem \ref{th3} covers all the cases when $\alpha_n$ grows at a subexponential rate, even those in which the sequence $\alpha_n$ oscillates very rapidly. For example, one can take $\alpha_n=2$ when $n$ is odd and $\alpha_n=n^p$ when $n$ is even and $p>0$.

If the sequence $\alpha_n$ is of exponential type, i.e. $\alpha_n=e^{cn}$ for every $n\in \mathbb{N}$, and hence there are no sharp oscillations, then Theorem \ref{th2} implies that the Hardy number is equal to infinity. However, if we allow wild oscillations and suppose that $\alpha_n \le e^n$, then the Hardy number might be finite. Actually, we construct such an example in Theorem \ref{ex}. Therefore, in order to obtain a general result in case $\alpha_n \le e^n$, we need to suppose that there are no wild oscillations. By imposing that
\[\mathop {\lim }\limits_{n \to  + \infty } \alpha_n =+\infty\] 
and thus preventing sharp oscillations of $\alpha_n$, we prove that ${\rm h}(C)=+\infty$. In fact, a more general result is true.

\begin{theorem}\label{gen}
	Let $\left\{ {b_n } \right\}_{n \in \mathbb{N}} $ be an increasing sequence of positive numbers such that $\mathop {\inf }\limits_{n>1} (b_n-b_{n-1})>0$. Let $\alpha_n \le e^{b_n}$ for every $n \in \mathbb{N}$. If 
	\[\mathop {\lim }\limits_{n \to  + \infty } \frac{{\log \alpha _n }}{{b_n  - b_{n - 1} }} = +\infty,\]
	then ${\rm h}(C)=+\infty$.
\end{theorem}

An interesting case, as we already remarked, is when $b_n=n$.

\begin{corollary}\label{cor}
	Let $\alpha_n \le e^n$ for every $n \in \mathbb{N}$. If 
	\[\mathop {\lim }\limits_{n \to  + \infty } \alpha_n = +\infty,\]
	then ${\rm h}(C)=+\infty$.
\end{corollary}
Next, we prove that the assumption in Theorem \ref{gen} is sharp. In other words, if there are wild oscillations of $\alpha_n$, then the Hardy number might be finite. 

\begin{theorem}\label{ex} Let $\left\{b_n \right\}$ be as Theorem \ref{gen}. There is a comb domain $C$ such that $\alpha_n \le e^{b_n}$ for every $n \in \mathbb{N}$, 
	\[\mathop {\liminf }\limits_{n \to  + \infty } \frac{{\log \alpha _n }}{{b_n  - b_{n - 1} }} < +\infty\]
and ${\rm h}(C)<+\infty$.
\end{theorem}

Theorem \ref{gen} covers a variety of cases such as $\alpha_n$ being comparable to $e^{n^{p}}$ for some $p>0$, $\alpha_n$ being comparable to $e^{e^{n^k}}$ for some $k<1$ and $\alpha_n$ being comparable to $e^{e^{n/\log n}}$. However, it does not apply if $\alpha_n$ is comparable to $e^{e^n}$. In this case, despite the fact that there are no wild oscillations of $\alpha_n$, Theorem \ref{th2} implies that  ${\rm h}(C)$ is finite. 

\begin{theorem}\label{last}
If $\alpha_n$ is comparable to $e^{e^n}$ for every $n \in \mathbb{N}$, then ${\rm h}(C)<+\infty$.
\end{theorem}

Therefore, the Hardy number of $C$ might be finite when the sequence $\alpha_n$ oscillates very quickly or if it goes to infinity rapidly enough like $\alpha_n$ being comparable to $e^{e^n}$.

\begin{remark}
Note that by (\ref{brownian}) all the results above concerning the Hardy number provide us with information about the finiteness of the moments of the exit time of Brownian motion from comb domains. In fact, ${\widetilde h}(D)$ is equal to infinity if and only if ${\rm h}(D)$ is equal to infinity.
\end{remark}

\begin{remark}
By Corollary \ref{cor} the Riemann mapping from $\mathbb{D}$ onto the comb domain $C$ with $x_n=e^n$ belongs to every $H^p(\mathbb{D})$ space. However, it does not belong to $BMOA$ (see \cite{St}). So, it is an example which ensures that
\[ BMOA \subsetneq \bigcap\limits_{p > 0} H^p(\mathbb{D}). \]
\end{remark}

In Section \ref{pre}, we introduce some preliminaries such as notions and results in hyperbolic geometry and their connection with the Hardy number. In Section \ref{condition}, we prove Theorem \ref{th2} and applying this, in Section \ref{proofs}, we prove all the other theorems stated above.

\section{Preliminary results}\label{pre}

\subsection{Hyperbolic distance}

The hyperbolic distance between two points $z,w$ in the unit disk $\mathbb{D}$ (see \cite[p.\ 11-28]{Bea}) is defined by 
\[{d_\mathbb{D}}\left( {z,w} \right) = \log \frac{{1 + \left| {\frac{{z - w}}{{1 - z\bar w}}} \right|}}{{1 - \left| {\frac{{z - w}}{{1 - z\bar w}}} \right|}}.\]
It can also be defined on any simply connected domain $D \ne \mathbb{C}$ in the following way: If $f$ is a Riemann mapping of $\mathbb{D}$ onto $D$ and $z,w \in D$, then 
${d_D}\left( {z,w} \right) = {d_\mathbb{D}}\left( {{f^{ - 1}}\left( z \right),{f^{ - 1}}\left( w \right)} \right)$.
Also, for a set $E \subset D$, we define ${d_D}\left( {z,E} \right)= \inf \left\{ {{d_D}\left( {z,w} \right):w \in E} \right\}$.

\subsection{Quasi-hyperbolic distance}

Let $D \ne \mathbb{C}$ be a simply connected domain. The hyperbolic distance between $z_1,z_2 \in D$ can be estimated by the quasi-hyperbolic distance which is defined by
\[{\delta _D}\left( {{z_1},{z_2}} \right) = \mathop {\inf }\limits_{\gamma :{z_1} \to {z_2}} \int_\gamma  {\frac{{\left| {dz} \right|}}{{d\left( {z,\partial D} \right)}}}, \]
where the infimum ranges over all the paths $\gamma$ connecting $z_1$ to $z_2$ in $D$ and $d\left( {z,\partial D} \right)$ denotes the Euclidean distance of $z$ from $\partial D$. It is known \cite[p. 33-36]{Bea} that 
\begin{equation}\label{qh}
\frac{1}{2}{\delta _D} \le {d_D} \le 2{\delta _D}.
\end{equation}
		
\subsection{Hardy number and hyperbolic distance}

In \cite{Kar} the current author proves that the Hardy number of a simply connected domain can be found with the aid of hyperbolic distance in the following way.

\begin{theorem}\label{kara} Let $D$ be a simply connected domain containing the origin. If ${F_r } =D\cap \{ |z|=r \}$ for $r >0$, then
	\[{\rm h}\left( D \right)=\mathop {\liminf}\limits_{r \to  + \infty } \frac{{{d_D}\left( {0,F_r} \right)}}{{\log r }}.\]
\end{theorem}

\subsection{The Stolz--Cesaro theorem}

Next, we state a generalized form of the Stolz--Cesaro theorem which we apply in Section \ref{proofs}. For the proof see \cite{Na} and \cite[p.\ 263--266]{Fur}.

\begin{theorem}\label{stolz}
	Let $\left\{b_n\right\}_{n\in\mathbb{N}}$ be a sequence of positive numbers such that $\sum\limits_{n = 1}^{ + \infty } {b_n }  =  + \infty $. For any real sequence $\left\{a_n\right\}_{n\in\mathbb{N}}$, it is true that
	\[\mathop {\limsup}\limits_{n \to  + \infty }\frac{a_1+a_2+\dots+a_n}{b_1+b_2+\dots+b_n}\le\mathop {\limsup}\limits_{n \to  + \infty }\frac{a_n}{b_n}\]
	and 
	\[\mathop {\liminf}\limits_{n \to  + \infty }\frac{a_1+a_2+\dots+a_n}{b_1+b_2+\dots+b_n}\ge\mathop {\liminf}\limits_{n \to  + \infty }\frac{a_n}{b_n}.\] 
\end{theorem}

\section{A necessary and sufficient condition}\label{condition}

In this section we give a necessary and sufficient condition for the Hardy number to be equal to infinity. First, we prove two auxiliary lemmas which give an upper and a lower estimate for ${\rm h}(C)$. 

\begin{lemma}\label{upper} Let $C$ be a comb domain of the form described in Section \ref{int}. If $K=4\log ((1+\sqrt{5})/2)$, then
		\begin{align}
		{\rm h}(C) \le \mathop {\lim \inf }\limits_{n \to  + \infty } \left( \frac{4{\sum\limits_{i = 1}^{n}\log  {\alpha_i} }}{{\log x_{n} }} + \frac{n{K}}{{\log x_{n} }}+4  \right). \nonumber
		\end{align}
\end{lemma}

\begin{proof} Let $r>0$. There exists a number $n \in \mathbb{N}$ such that $x_{n-1}<r \le x_n$. Due to the symmetry of $C$ with respect to the real axis and the uniqueness of the hyperbolic geodesic in simply connected domains, the hyperbolic geodesic between $0$ and $r$ in $C$ is the line segment $[0,r]$. Therefore, we have 
	\begin{equation}\label{sx4}
	d_C (0, r)=d_C (0, x_{n-1})+d_C (x_{n-1}, r)
	\end{equation}
	(see \cite[p. 14]{Bea}). Applying (\ref{qh}) and letting $m_{i-1}$ denote the midpoint of the interval $\left[ {x_{i-1} ,x_i } \right]$ for every $i \in \mathbb{N}$, we infer that
	\begin{align}\label{sx5}
	d_C \left( {0 ,x_{n-1} } \right) &\le 2\delta _C \left( {0 ,x_{n-1} } \right) = 2\int_{0 }^{x_{n-1} } {\frac{{dx}}{{d\left( {x,\partial C} \right)}}}  = 2\sum\limits_{i = 1}^{n-1} {\int_{x_{i-1} }^{x_{i} }  {\frac{{dx}}{{d\left( {x,\partial C} \right)}}} } \nonumber \\
	&=4 \sum\limits_{i = 1}^{n-1} {\int_{x_{i-1} }^{m_{i-1} } {\frac{{dx}}{{d\left( {x,\partial C} \right)}}} }=4 \sum\limits_{i = 1}^{n-1} {\int_{x_{i-1} }^{m_{i-1} } {\frac{{dx}}{{\sqrt {1  + \left( {x - x_{i-1} } \right)^2 } }}} } \nonumber \\
	&=4 \sum\limits_{i = 1}^{n-1} {\arcsinh \left( {{{m_{i-1} - x_{i-1} }}} \right)}=4 \sum\limits_{i = 1}^{n-1} {\arcsinh \left( {\frac{{x_{i}  - x_{i-1} }}{{2}}} \right)} \nonumber \\
	&=4\sum\limits_{i = 1}^{n-1} {\log \left( {\frac{\alpha_i}{{2}} + \sqrt {\left( {\frac{\alpha_i}{{2}}} \right)^2  + 1} } \right)}. 
	\end{align}
	Recall that the domain $C$ has the property that $ \mathop {\inf }\limits_{n \in \mathbb{N}} \alpha_n>1$,
	which implies that, for every $i\in \mathbb{N}$,
	\[\frac{\alpha_i}{{2}} > \frac{1}{{2}}.\]
	Therefore,
	\begin{equation}\nonumber
	\sqrt {\left( {\frac{\alpha_i}{{2}}} \right)^2  + 1} = \sqrt {\left( {\frac{\alpha_i}{{2}}} \right)^2  +4 \left( {\frac{1}{{2}}} \right)^2 }\le  \frac{\alpha_i}{{2}}\sqrt {5}
	\end{equation}
	and hence
	\begin{align}\label{sx6}
	\log \left( {\frac{\alpha_i}{{2}} +  \sqrt {\left( {\frac{\alpha_i}{{2}}} \right)^2 + 1} } \right) \le \log \left( {\frac{\alpha_i}{{2}}\left( {1 + \sqrt {5} } \right)} \right)=\log {{\alpha_i}} + \frac{K}{4},
	\end{align}
	where $K=4\log ((1+\sqrt{5})/2)$. Combining (\ref{sx5}) with (\ref{sx6}), we deduce that
	\begin{equation}\label{p1}
	d_C \left( {0 ,x_{n-1} } \right) \le 4\sum\limits_{i = 1}^{n-1} {\log {{\alpha_i}}}  + (n-1){K}.
	\end{equation}
	Now, in order to find an upper estimate for $d_C (x_{n-1},r)$, we consider the following cases.

Case 1: If $r \in \left( {x_{n-1} ,\frac{{x_{n-1}  + x_n }}{2}} \right]$, then 
	\begin{align}
	d_C (x_{n-1},r) &\le 2 \delta_C (x_{n-1},r)=2\int_{x_{n-1} }^r  {\frac{{dx}}{{\sqrt {1 + (x - x_{n-1} )^2 } }}} \nonumber \\
	&=2 \arcsinh \left( {r  - x_{n-1}} \right)\le 2 \arcsinh  {r}. \nonumber
	\end{align}
	
	Case 2: If $r \in \left( {\frac{{x_{n-1}  + x_{n} }}{2},x_{n} } \right]$, then 
	\begin{align}
	d_C (x_{n-1},r) &\le d_C (x_{n-1},x_{n}) \le 2 \delta_C (x_{n-1},x_n)=2\int_{x_{n-1}}^{x_n}  {\frac{{dx}}{{d(x,\partial C)}}}  \nonumber \\
	&= 4\int_{x_{n-1} }^{\frac{{x_{n-1} + x_n }}{2}} {\frac{{dx}}{{\sqrt {1 + (x - x_{n-1} )^2 } }}} = 4\arcsinh \left( {\frac{{x_{n}  - x_{n-1} }}{2}} \right)  \nonumber \\
	&\le 4\arcsinh  {r}. \nonumber
	\end{align}
	Therefore, it follows that in both cases,
	\begin{equation}\label{pp3}
	d_C (x_{n-1} ,r) \le 4\arcsinh  {r}.
	\end{equation}
	 Recall that $r>x_{n-1}$. So, by (\ref{sx4}), (\ref{p1}) and (\ref{pp3}) we derive that
	\[\frac{{d_C (0,r)}}{{\log r }} \le \frac{4\sum\limits_{i = 1}^{n-1} {\log \alpha_i}  }{{\log x_{n-1} }}+ \frac{(n-1){K}}{\log x_{n-1} } + \frac{4\arcsinh {r}}{{\log r }}.\]
	This in conjunction with Theorem \ref{kara} gives
\begin{align}
{\rm h}(C) \le \mathop {\lim \inf }\limits_{r \to  + \infty } \frac{d_C (0,r)}{\log r} \le \mathop {\lim \inf }\limits_{n \to  + \infty } \left( \frac{4{\sum\limits_{i = 1}^{n}\log  {\alpha_i} }}{{\log x_{n} }} + \frac{n{K}}{{\log x_{n} }}+4  \right) \nonumber
\end{align}
and the proof is complete.
\end{proof}

\begin{lemma}\label{lower} Let $C$ be a comb domain of the form described in Section \ref{int}. Then
	\begin{align}
	{\rm h}(C) \ge \mathop {\lim \inf }\limits_{n \to  + \infty } \frac{{\sum\limits_{i = 1}^{n}\log  {\alpha_i} }}{{\log x_{n} }}-1. \nonumber
	\end{align}
\end{lemma}	

\begin{proof} If ${F_r } =C\cap \{ |z|=r \}$, by Theorem \ref{kara} we have
\begin{align}
{\rm h}(C) = \mathop {\lim \inf }\limits_{r \to  + \infty }\frac{d_C (0,F_r)}{\log r}=\mathop {\lim }\limits_{n \to  + \infty }\frac{d_C (0,F_{r_n})}{\log r_n}, \nonumber
\end{align}
where $\left\{r_n\right\}$ is an increasing sequence of positive numbers. The hyperbolic distance $d_C (0,F_{r_n})$ is attained on some component of $F_{r_n}$ lying in the vertical strip $\left\{z:x_{i_n}<{\rm{Re}} z<x_{{i_n}+1} \right\}$, where $\{x_{i_n}\}$ is a subsequence of $\{x_n\}$. If we denote this component by $F_{r_n}^{x_{i_n}}$ then
\begin{align}\label{isot}
{\rm h}(C)=\mathop {\lim }\limits_{n \to  + \infty }\frac{d_C (0,F_{r_n}^{x_{i_n}})}{\log r_n}.
\end{align}
Since $C$ is symmetric with respect to the real axis, without loss of generality, we suppose that $F_{r_n}^{x_{i_n}}$ lies on the upper half-plane or intersects the positive real axis (see Fig.\ \ref{fig}). If ${\rm h}(C)=+\infty$, the result is trivial. Hence, we suppose that ${\rm h}(C)<+\infty$ and take the following cases.

\begin{figure} 
	\begin{center}
		\includegraphics[scale=0.6]{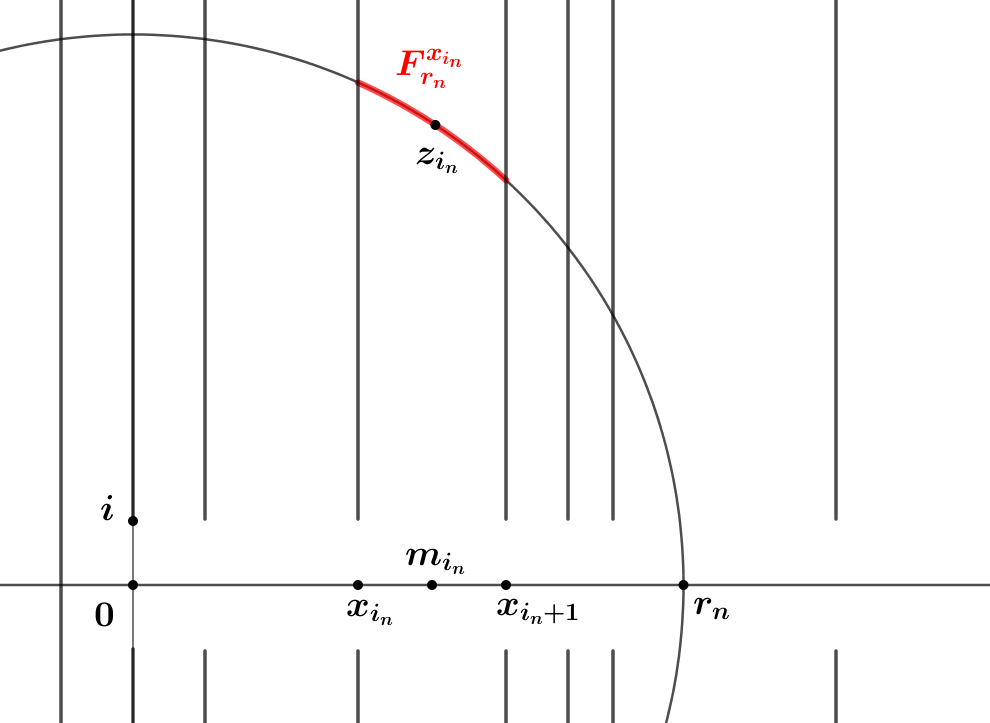}
		\vspace*{2mm}
		\caption{The component $F_{r_n}^{x_{i_n}}$.}
		\label{fig}
	\end{center}
\end{figure}

Case 1: For infinitely many $n$, $F_{r_n}^{x_{i_n}}$ is the component of $F_{r_n}$ containing $r_n$. By passing to a subsequence we assume that this is the case for all $n$. The hyperbolic geodesic between $0$ and $F_{r_n}^{x_{i_n}}$ passes from some point $u_{i_n}$ of the line segment $\{x_{i_n}+iy:-1<y<1\}$. If $m_{j-1}$ denotes the midpoint of the interval $[x_{j-1},x_j]$, then we have  
\begin{align}\label{sx2}
d_C \left( {0 ,F_{r_n}^{x_{i_n}} } \right) &\ge d_C \left( {0 ,u_{i_n}} \right)\ge \frac{1}{2}\delta _C \left( {0 , u_{i_n} } \right) \ge \frac{1}{2}\int_{0}^{x_{i_n} } {\frac{{dx}}{{d\left( {x,\partial C} \right)}}} \nonumber \\
&=\frac{1}{2}\sum\limits_{j = 1}^{i_n} {\int_{x_{j-1} }^{x_j } {\frac{{dx}}{{d\left( {x,\partial C} \right)}}} }=\sum\limits_{j = 1}^{i_n} {\int_{x_{j-1} }^{m_{j-1} } { {\frac{{dx}}{{\sqrt {1  + \left( {x - x_{j-1} } \right)^2 } }}} }} \nonumber\\	
&=\sum\limits_{j = 1}^{i_n} {\arcsinh \left(\frac{{{x_{j}  - x_{j-1} }}}{2} \right)} =\sum\limits_{j = 1}^{i_n} {\log \left( {\frac{\alpha_j}{2} + \sqrt {\left(\frac{\alpha_j}{2}\right)^2  + 1} } \right)} \nonumber \\
&\ge\sum\limits_{j = 1}^{i_n} {\log \alpha_j}.
\end{align}
Since $x_{i_n}<r_n<x_{i_n+1}$, by (\ref{sx2}) and (\ref{isot}) it follows that
\begin{align}\label{way1}
{\rm h}(C) &\ge \mathop {\lim \inf }\limits_{n \to  + \infty } \frac{\sum\limits_{j = 1}^{i_n} {\log \alpha_j}}{{\log x_{i_n+1} }}=\mathop {\lim \inf }\limits_{n \to  + \infty } \left(\frac{\sum\limits_{j = 1}^{i_n+1} {\log \alpha_j}}{{\log x_{i_n+1} }}-\frac{\log \alpha_{i_n+1}}{\log x_{i_n+1}}\right) \nonumber \\
&\ge\mathop {\lim \inf }\limits_{n \to  + \infty }\frac{\sum\limits_{j = 1}^{i_n+1} {\log \alpha_j}}{{\log x_{i_n+1} }}-1\ge\mathop {\lim \inf }\limits_{n \to  + \infty } \frac{\sum\limits_{j = 1}^{n} {\log \alpha_j}}{{\log x_{n} }}-1.
\end{align}

Case 2: For infinitely many $n$, $F_{r_n}^{x_{i_n}}$ is not the component of $F_{r_n}$ containing $r_n$. By passing to a subsequence we suppose that this is the case for all $n$. First, suppose that for $k=2/(3{\rm h}(C))$,
\begin{align}\label{case1}
x_{{i_n}+1}\le k\frac{r_n}{\log r_n }
\end{align}
for infinitely many $n$. The hyperbolic geodesic between $0$ and $F_{r_n}^{x_{i_n}}$ passes from some point $v_{i_n}$ of the line segment $(x_{i_n}+i,x_{i_n+1}+i)$ and some point $z_{i_n}$ of the line segment $(x_{i_n}+i\sqrt{r_n^2-x_{i_n+1}^2},x_{i_n+1}+i\sqrt{r_n^2-x_{i_n+1}^2})$. Thus, we have   
\begin{align}
d_C (0,F_{r_n}^{x_{i_n}})&\ge d_C (v_{i_n},z_{i_n})\ge \frac{1}{2}\delta_C (v_{i_n},z_{i_n})\ge \frac{1}{2}\int_1^{\sqrt{r_n^2-x_{i_n+1}^2}} \frac{dx}{d(x,\partial C)} \nonumber \\
&=\frac{\sqrt{r_n^2-x_{i_n+1}^2}-1}{x_{{i_n}+1}-x_{i_n}}\ge \frac{\sqrt{r_n^2-x_{i_n+1}^2}-1}{x_{{i_n}+1}}  \nonumber \\
&\ge \frac{1}{k}\frac{\log r_n}{r_n}\left(\sqrt{r_n^2-k^2\left( \frac{r_n}{\log r_n}\right)^2}-1\right), \nonumber
\end{align}
where we applied (\ref{case1}). This in combination with (\ref{isot}) implies that
\[{\rm h}(C)\ge\frac{1}{k}\mathop {\lim \inf }\limits_{n \to  + \infty }\frac{\sqrt{r_n^2-k^2\left( \frac{r_n}{\log r_n}\right)^2}-1}{r_n} =\frac{1}{k}=\frac{3}{2}{\rm h}(C),\]
which is a contradiction. Therefore,
\begin{align}\label{case2}
x_{{i_n}+1}> k\frac{r_n}{\log r_n }
\end{align}
for all but finitely many $n$. So, working as in Case 1, we have
\begin{align}\nonumber
d_C (0,F_{r_n}^{x_{i_n}})\ge \sum\limits_{j = 1}^{i_n} {\log \alpha_j}.
\end{align}
By this and (\ref{isot}), it follows that
\begin{align}\nonumber
{\rm h}(C) &\ge \mathop {\lim \inf }\limits_{n \to  + \infty } \frac{\sum\limits_{j = 1}^{i_n} {\log \alpha_j}}{{\log r_n }}=\mathop {\lim \inf }\limits_{n \to  + \infty } \left(\frac{\sum\limits_{j = 1}^{i_n} {\log \alpha_j}}{{\log x_{i_n+1}}}\frac{\log x_{i_n+1}}{\log r_n} \right) \nonumber \\
&\ge \mathop {\lim \inf }\limits_{n \to  + \infty } \left(\frac{\sum\limits_{j = 1}^{i_n} {\log \alpha_j}}{{\log x_{i_n+1}}}\frac{\log k +\log r_n -\log (\log r_n)}{\log r_n} \right) \nonumber \\
&=\mathop {\lim \inf }\limits_{n \to  + \infty }\frac{\sum\limits_{j = 1}^{i_n} {\log \alpha_j}}{{\log x_{i_n+1}}}\ge\mathop {\lim \inf }\limits_{n \to  + \infty } \frac{\sum\limits_{j = 1}^{n} {\log \alpha_j}}{{\log x_{n} }}-1, \nonumber 
\end{align}
where we applied (\ref{case2}) and (\ref{way1}). Consequently, in any case we obtain the desired result.
\end{proof}

Next, we prove our main theorem.

\begin{proof}[Proof of Theorem \ref{th2}]
	Suppose that ${\rm h}(C)=+\infty$. If 
	\[\mathop {\lim \inf }\limits_{n \to  + \infty } \frac{n}{{\log x_{n} }}<+\infty,\]
	then by Lemma \ref{upper} we deduce that
	\[\mathop {\lim}\limits_{n \to  + \infty } \frac{{\sum \limits_{i = 1}^{n} \log  {{{ \alpha_i }}}}}{{\log x_{n} }}=+\infty.\]
	Now, suppose that
	\[\mathop {\lim \inf }\limits_{n \to  + \infty } \frac{n}{{\log x_{n} }}=+\infty.\]
	Recall that $\mathop {\inf }\limits_{n \in \mathbb{N}} \alpha_n = l > 1$. So, we have
	\[\frac{{\sum \limits_{i = 1}^{n} \log  {{{\alpha_i }}}}}{{\log x_{n} }}> \frac{n}{\log x_{n} }\log l\]
	and thus
	\[\mathop {\lim }\limits_{n \to  + \infty } \frac{{\sum \limits_{i = 1}^{n} \log  {{{ \alpha_i }}}}}{{\log x_{n} }}=+\infty\]
	in both cases. The other direction is direct by Lemma \ref{lower}.  
\end{proof}

\section{Consequent results}\label{proofs}

In this section we prove several results  derived by Theorem \ref{th2}. They are all stated in Section \ref{int}. First, we show that if the sequence $\alpha_n$ grows at a subexponential rate, then the Hardy number is equal to infinity.  

\begin{proof}[Proof of Theorem \ref{th3}]
Let $A_n  = \mathop {\max }\limits_{1 \le j \le n} \alpha _j$. First, we prove that our assumption implies that
\[\mathop {\lim}\limits_{n \to  + \infty }\frac{\log A_n}{n}=0.\]
Suppose, on the contrary, that it is false. Then there are a constant $\delta>0$ and  a subsequence $\left\{ {A_{k_n} } \right\}_{_{n \in \mathbb{N}} } $ of $\left\{ {A_n } \right\}_{_{n \in \mathbb{N}} } $ such that, for every $n \in \mathbb{N}$,
\[\frac{\log A_{k_n}}{k_n}\ge \delta.\]
For every $n \in \mathbb{N}$ there is an $m_n \in \mathbb{N}$  such that $A_{k_n}=\alpha_{m_n}$ and $1\le m_n \le k_n$.

Case 1: If there is a constant $K>0$ such that $m_n \le K$ for every $n \in \mathbb{N}$, then
\[0\le \frac{\log A_{k_n}}{k_n}= \frac{\log \alpha_{m_n}}{k_n} \le \frac{\mathop {\max }\limits_{1 \le i \le K} \log \alpha _i }{k_n}. \] 
So, taking limits as $n \to +\infty$, we derive that
\[\mathop {\lim}\limits_{n \to  + \infty }\frac{\log A_{k_n}}{k_n}=0,\]
which is a contradiction.
 
Case 2: If $m_n \to +\infty$, then there is a subsequence $\left\{ {{m_{l_n}} } \right\}_{_{n \in \mathbb{N}} } $ such that $m_{l_n} \to +\infty$ and $m_{l_n}$ is strictly increasing with respect to $n$. Thus, 
\[\delta \le \frac{\log A_{k_{l_n}}}{k_{l_n}} = \frac{\log \alpha_{m_{l_n}}}{k_{l_n}}=\frac{\log \alpha_{m_{l_n}}}{m_{l_n}} \frac{{m_{l_n}}}{k_{l_n}}\le \frac{\log \alpha_{m_{l_n}}}{m_{l_n}}. \]
Taking limits as $n \to +\infty$, we infer that $\delta \le 0$, which is a contradiction. Therefore,
\begin{equation}\label{ap}
\mathop {\lim}\limits_{n \to  + \infty }\frac{\log A_n}{n}=0.
\end{equation} 
Recall that ${\inf }\alpha_n=l>1$. Since $\alpha _n  \le A_n $ for every $n \in \mathbb{N}$ and $\left\{ {A_n } \right\}_{_{n \in \mathbb{N}} } $ is an increasing sequence, we have
\begin{equation} \nonumber
\frac{{\sum \limits_{i = 1}^{n} \log  {{{  {\alpha_{i} } }}}}}{{\log \sum \limits_{i = 1}^{n}  {{{  {\alpha_{i} } }}} }}\ge \frac{n \log l}{\log \sum \limits_{i = 1}^{n}  {{{  {A_{i} } }}}} \ge \frac{n \log l}{\log n + \log A_n}=\frac{\log l}{\frac{\log n}{n}+\frac{\log A_n}{n}}.
\end{equation}
Taking limits as $n \to +\infty$, by (\ref{ap}) we deduce that
\[\mathop {\lim }\limits_{n \to  + \infty } \frac{{\sum \limits_{i = 1}^{n} \log  {{{  {\alpha_{i} } }}}}}{{\log \sum \limits_{i = 1}^{n}  {{{  {\alpha_{i} } }}} }} =+\infty.\]
Thus, Theorem \ref{th2} implies that ${\rm h}(C)=+\infty$.
\end{proof}

The following corollary of Theorem \ref{th3} is the corollary of Theorem 4 in \cite[p.\ 3]{Bou}. Let $\left\{x_n\right\}_{n\in \mathbb{Z}}$ be an increasing sequence of distinct real numbers without accumulation point in $\mathbb{R}$ and $\left\{c_n\right\}_{n\in \mathbb{Z}}$ be an associated sequence of positive numbers, and let
\[D=\mathbb{C}\backslash \bigcup\limits_{n \in \mathbb{Z}} {\left\{ {x_n  + iy_n: |y_n| \ge c_n} \right\}}.\]

\begin{corollary}
Let  $\alpha_n=x_n-x_{n-1}$. Suppose that $
\mathop {\inf }\limits_{n \in \mathbb{Z}} \alpha_n >0$ and $\left\{ c_n \right\}_{n \in \mathbb{Z}}$ is bounded. If
	\[\sum\limits_{j = 1}^{ + \infty } {(\mathop {\max }\limits_{|n| \le j} \alpha _n^2 )\theta ^j }  <  + \infty \]
	for every $\theta \in (0,1)$, then ${\widetilde h}(D)=+\infty$.
\end{corollary}

\begin{proof}
	By assumption, there is a constant $c>0$ such that $c_n\le c$ for every $n\in \mathbb{Z}$. Thus,
	\[D\subseteq \mathbb{C}\backslash \bigcup\limits_{n \in \mathbb{Z}} {\left\{ {x_n  + iy: |y| \ge c} \right\}}:=D_c\]
	and ${\rm h}(D)\ge {\rm h}(D_c)$. By this and (\ref{brownian}), it suffices to prove that ${\rm h}(D_c)=+\infty$ or, equivalently, ${\rm h}(D_c^-)={\rm h}(D_c^+)=+\infty$ (see Section \ref{int}). Without loss of generality, we suppose that $c=1$ and $\inf \alpha_n >1$. We have
	\[ \sum\limits_{j = 1}^{ + \infty } {\alpha _j^2 \theta ^j } \le \sum\limits_{j = 1}^{ + \infty } {(\mathop {\max }\limits_{|n| \le j} \alpha _n^2 )\theta ^j }  <  + \infty \]
	for every $\theta \in (0,1)$. This implies that, for every $\theta \in (0,1)$,
	\[\mathop {\lim }\limits_{n \to  + \infty } \alpha _n^2 \theta ^n  = 0\]
	and hence for every $\theta \in (0,1)$ there is an $n_0(\theta) \in \mathbb{N}$ such that for $n\ge n_0$,
	\[\alpha _n \theta ^{n/2} <1\]
	or, equivalently,
	\[\frac{\log \alpha_n}{n}<\frac{1}{2}\log \frac{1}{\theta}.\] 
	Set $\varepsilon=(1/2)\log (1/\theta)$. So, for every $\varepsilon>0$ there is an $n_0 (\varepsilon) \in \mathbb{N}$ such that for $n\ge n_0$,
	\[\frac{\log \alpha_n}{n}<\varepsilon.\]
	By Theorem \ref{th3}, we deduce that ${\rm h}(D_c^+)=+\infty$. Working with $\alpha_{n}$ for $n<0$ in the same way as above, we infer that ${\rm h}(D_c^-)=+\infty$ and thus, it follows that ${\rm h}(D_c)=+\infty$. 
\end{proof}

Next, we prove Theorem \ref{gen} which implies that if the sequence $\log \alpha_n$ grows at a subexponential rate and there are no wild oscillations of $\alpha_n$, then the Hardy number is equal to infinity.

\begin{proof}[Proof of Theorem \ref{gen}]
	Since $\left\{ {b_n } \right\}_{n \in \mathbb{N}} $ is an increasing sequence, we have
	\begin{align}\label{logsum}
	\log \sum \limits_{i = 1}^{n}  {{{  {\alpha_{i} } }}} \le \log \sum \limits_{i = 1}^{n}  {e^{b_i}}\le \log \left( ne^{b_n} \right) =\log n +b_n.
	\end{align}
	By assumption, $\mathop {\inf }\limits_{n>1} (b_n-b_{n-1})=r>0$. This implies that
	\begin{equation}\label{nbn}
	b_n=\sum \limits_{i = 2}^{n}(b_i-b_{i-1})+b_1\ge (n-1)r+b_1.
	\end{equation}
	So, for every $n \in \mathbb{N}$, we have
	\[0\le \frac{\log n}{b_n} \le \frac{\log n}{(n-1)r+b_1}\]
	and thus
	\[\mathop {\lim }\limits_{n \to  + \infty }\frac{\log n}{b_n}=0.\]
	By this and (\ref{logsum}), we obtain the following estimates
	\begin{align}
	\mathop {\liminf }\limits_{n \to  + \infty }  \frac{{\sum \limits_{i = 1}^{n} \log  {{{  {\alpha_{i} } }}}}}{{\log \sum \limits_{i = 1}^{n}  {{{  {\alpha_{i} } }}} }}&\ge \mathop {\liminf }\limits_{n \to  + \infty } \left( \frac{\sum \limits_{i = 1}^{n} \log  {{{  {\alpha_{i} } }}}}{b_n}\frac{b_n}{\log n+b_n}\right) \nonumber \\
	&=\mathop {\liminf }\limits_{n \to  + \infty }  \frac{\sum \limits_{i = 1}^{n} \log  {{{  {\alpha_{i} } }}}}{b_1+\sum \limits_{i = 2}^{n} {{{ (b_i-b_{i-1}) }}}}  \nonumber \\
	&\ge \mathop {\liminf }\limits_{n \to  + \infty } \frac{\log \alpha_n}{b_n-b_{n-1}}=+\infty. \nonumber
	\end{align}
	In the last inequality we applied Theorem \ref{stolz}. Therefore, Theorem \ref{th2} implies that ${\rm h}(C)=+\infty$.
\end{proof}

Next, we prove that the condition of Theorem \ref{gen} is sharp. 

\begin{proof}[Proof of Theorem \ref{ex}]
Fix a $c>1$ and let $\left\{ {b_{k_m } } \right\}_{m \in \mathbb{N}} $ be a subsequence of $\left\{ {b_{m} } \right\}_{m \in \mathbb{N}} $ such that
\begin{equation}\label{sum}
b_{k_m }  \ge \sum\limits_{i = 1}^{m - 1} {b_{k_i } }
\end{equation} 
for every $m \ge 2$. Moreover, we observe that (\ref{nbn}) implies that
\begin{align}\label{bkm}
\frac{k_m}{b_{k_m}} \le \frac{1}{r}+\frac{r-b_1}{rb_{k_m} }.
\end{align}
We consider a comb domain with 
\[\alpha _n  = \left\{ \begin{array}{l}
c,\,\,n \notin \left\{k_m:m\in \mathbb{N} \right\}  \\ 
e^{b_{k_m } } ,\,\,n=k_m\,\,{\text {for some} }\,\,m\in\mathbb{N} \\ 
\end{array} \right..\]
Applying (\ref{sum}) and (\ref{bkm}), we have the following estimates 
\begin{align}
I_{k_m }  &:= \frac{{\sum\limits_{i = 1}^{k_m} {\log \alpha _i } }}{{\log \sum\limits_{i = 1}^{k_m} {\alpha _i } }}=\frac{{(k_m  - m)\log c + \sum\limits_{i = 1}^m {b_{k_i } } }}{{\log \left( {(k_m  - m)c + \sum\limits_{i = 1}^m {e^{b_{k_i } } } } \right)}} \nonumber \\
&\le \frac{k_m \log c +2b_{k_m}}{b_{k_m}} \le \left(\frac{1}{r}+\frac{r-b_1}{rb_{k_m} }\right)\log c + 2. \nonumber
\end{align} 
This implies that
\[\mathop {\liminf }\limits_{n \to  + \infty } \frac{{\sum \limits_{i = 1}^{n} \log  {{{  {\alpha_{i} } }}}}}{{\log \sum \limits_{i = 1}^{n}  {{{  {\alpha_{i} } }}} }} \le \mathop {\lim \inf }\limits_{m \to  + \infty } I_{k_m } \le 2+\frac{1}{r} \log c\]
and hence by Theorem \ref{th2} we derive that ${\rm h}(C)<+\infty$. Finally, it follows that
\begin{align}\nonumber
\mathop {\liminf }\limits_{n \to  + \infty } \frac{{\log \alpha _n }}{{b_n  - b_{n - 1} }} \le \mathop {\liminf }\limits_{m \to  + \infty } \frac{{\log \alpha _{k_m+1} }}{{b_{k_m+1}  - b_{k_m} }}\le \frac{\log c}{r}<+\infty 
\end{align}
and the proof is complete.
\end{proof}

Finally, we prove that if $\alpha_n$ is comparable to $e^{e^n}$, then the Hardy number is finite.

\begin{proof}[Proof of Theorem \ref{last}]
By assumption there are constants $c_1,c_2>0$ such that, for every $n\in \mathbb{N}$,
\[c_1e^{e^n}\le\alpha_n\le c_2e^{e^n}.\]
So, it follows that
\begin{align}
\mathop {\liminf }\limits_{n \to  + \infty } \frac{{\sum \limits_{i = 1}^{n} \log  {{{  {\alpha_{i} } }}}}}{{\log \sum \limits_{i = 1}^{n}  {{{  {\alpha_{i} } }}} }}&\le \mathop {\liminf }\limits_{n \to  + \infty } \frac{{\sum \limits_{i = 1}^{n} \log  {{{  {\alpha_{i} } }}}}}{{\log \alpha_n }}\le \mathop {\liminf }\limits_{n \to  + \infty }\frac{n\log c_2+\sum \limits_{i = 1}^{n} e^i }{\log c_1 +e^n} \nonumber \\
&=\frac{e}{e-1}\mathop {\liminf }\limits_{n \to  + \infty }\frac{e^n}{\log c_1+e^n}=\frac{e}{e-1}<+\infty. \nonumber 
\end{align}
By Theorem \ref{th2} we deduce that ${\rm h (C)}<+\infty$.	
\end{proof}


\begin{bibdiv}
\begin{biblist}
	
\bib{Ba}{article}{
	title={Univalent functions, Hardy spaces and spaces of Dirichlet type},
	author={A. Baernstein and D. Girela and J. \'{A}. Pel\'{a}ez,},
	journal={Illinois J. Math.},
	volume={48},
	date={2004},
	pages={837--859}
}
	
\bib{Bea}{article}{
	title={The hyperbolic metric and geometric function theory},
	author={A.F. Beardon and D. Minda,},
	journal={Quasiconformal mappings and their applications},
	date={2007},
	pages={9--56}
}

\bib{Bet}{article}{
	title={Harmonic measure on simply connected domains of fixed inradius},
	author={D. Betsakos},
	journal={Ark. Mat.},
	volume={36},
	date={1998},
	pages={275--306}
}
\bib{Betsa}{article}{
	title={On the asymptotic behavior of the trajectories of semigroups of holomorphic functions},
	author={D. Betsakos},
	journal={J. Geometric Analysis},
	volume={26},
	date={2016},
	pages={557--569}
}
\bib{Betsakos}{article}{
	title={On the probability of fast exits and long stays of a planar Brownian motion in simply connected domains},
	author={D. Betsakos and M. Boudabra and G. Markwosky},
	journal={J. Math. Anal. Appl.},
	volume={493},
	date={2021},
	pages={10 pp}
}
\bib{Bou}{article}{
	title={On the finiteness of moments of the exit time of planar Brownian motion from comb domains},
	author={M. Boudabra and G. Markowsky},
	journal={Ann. Fenn. Math.},
	volume={46},
	date={2021},
	pages={527--536}
}
\bib{Bur}{article}{
	title={Exit times of Brownian motion, harmonic majorization, and Hardy spaces},
	author={D. L. Burkholder},
	journal={Advances in Mathematics},
	volume={26},
	date={1977},
	pages={182--205}
}
\bib{Dur}{book}{
	title={Theory of $H^p$ Spaces},
	author={P.L. Duren},
	date={1970},
	publisher={Academic Press},
	address={New York-London}
}

\bib{Ess}{article}{
	title={On analytic functions which are in $H^p$ for some positive $p$},
	author={M. Ess{\' e}n},
	journal={Ark. Mat.},
	volume={19},
	date={1981},
	pages={43--51}
}

\bib{Fur}{book}{
	title={Limits, Series, and Fractional Part Integrals},
	author={O. Furdui},
	series={Problems in mathematical analysis, Problem Books in Mathematics},
	date={2013},
	publisher={Springer},
	address={New York}
}
\bib{Han1}{article}{
	title={Hardy classes and ranges of functions},
	author={L.J. Hansen},
	journal={Michigan Math. J.},
	volume={17},
	date={1970},
	pages={235--248}
}
\bib{Han2}{article}{
	title={The Hardy class of a spiral-like function},
	author={L.J. Hansen},
	journal={Michigan Math. J.},
	volume={18},
	date={1971},
	pages={279--282}
}

\bib{Jen}{article}{
	title={On comb domains},
	author={J.A. Jenkins},
	journal={Proc. Amer. Math. Soc.},
	volume={124},
	date={1996},
	pages={187--191}
}
\bib{Kar1}{article}{
	title={Hyperbolic distance and membership of conformal maps in the Hardy space},
	author={C. Karafyllia},
	journal={Proc. Amer. Math. Soc.}
	volume={147},
	date={2019},
	pages={3855--3858}
}

\bib{Kar}{article}{
	title={On the Hardy number of a domain in terms of harmonic measure and hyperbolic distance},
	author={C. Karafyllia},
	journal={Ark. Mat.},
	volume={58},
	date={2020},
	pages={307--331}

}
\bib{Karam}{article}{
	title={On the Angular Derivative of Comb Domains},
	author={N. Karamanlis},
	journal={Comput. Methods Funct. Theory},
	volume={19},
	date={2019},
	pages={613-–623}
}

\bib{Kim}{article}{
	title={Hardy spaces and unbounded quasidisks},
	author={Y.C. Kim and T. Sugawa},
	journal={Ann. Acad. Sci. Fenn. Math.},
	volume={36},
	date={2011},
	pages={291--300}
}
\bib{Na}{article}{
	title={The Stolz-Cesaro theorem},
	author={G. Nagy},
	journal={Manuscript available electronically at https://www.math.ksu.edu/~nagy/snippets/stolz-cesaro.pdf},
	volume={},
	date={Accessed on 01/21/2021},
	pages={1--4}
}
\bib{Ra}{article}{
	title={Univalent functions in Hardy, Bergman, Bloch and related spaces},
	author={F. P\'{e}rez-Gonz\'{a}lez and J. R\"{a}tty\"{a}},
	journal={J. d' Anal. Math.},
	volume={105},
	date={2008},
	pages={125--148}
}
\bib{Pog}{article}{
	title={Hardy spaces and twisted sectors for geometric models},
	author={P. Poggi-Corradini},
	journal={Trans. Amer. Math. Soc.},
	volume={348},
	date={1996},
	pages={2503--2518}
}

\bib{Co2}{article}{
	title={The Hardy class of K{\oe}nigs maps},
	author={P. Poggi-Corradini},
	journal={Michigan Math. J.},
	volume={44},
	date={1997},
	pages={495--507}
}


\bib{St}{article}{
	title={A geometric characterization of analytic functions with bounded mean oscillation},
	author={D. Stegenga and K. Stephenson},
	journal={J. London Math. Soc (2)},
	volume={24},
	date={1981},
	pages={243--254}
}
\end{biblist}
\end{bibdiv}
\end{document}